\documentclass[12pt,a4paper,reqno]{amsart}
\usepackage[utf8]{inputenc}
\usepackage[T1]{fontenc}
\usepackage{amsmath}
\usepackage{amsthm}
\usepackage{amssymb}
\usepackage[abbrev]{amsrefs}
\usepackage{mathrsfs}
\usepackage[dvipsnames]{xcolor}
\usepackage{bm}
\usepackage{bbm}
\usepackage{enumitem}
\usepackage{hyperref}
\AtBeginDocument{\def\MR#1{}}
\makeatletter
\@namedef{subjclassname@2020}{%
\textup{2020} Mathematics Subject Classification}
\makeatother
\newtheorem{thmintro}{}

\newtheorem{theoremintro}[thmintro]{Theorem}
\theoremstyle{definition}
\newtheorem{questionintro}[thmintro]{Question}
\newtheorem{thm}{}[section]
\newtheorem{theorem}[thm]{Theorem}
\newtheorem{corollary}[thm]{Corollary}
\newtheorem{lemma}[thm]{Lemma}
\newtheorem{proposition}[thm]{Proposition}
\newtheorem{question}[thm]{Question}
\theoremstyle{definition}
\newtheorem{definition}[thm]{Definition}
\newtheorem{example}[thm]{Example}
\numberwithin{equation}{section}
\allowdisplaybreaks
\newcommand{\Nnorm}[1]{{\left\vert\kern-0.25ex\left\vert\kern-0.25ex\left\vert #1
\right\vert\kern-0.25ex\right\vert\kern-0.25ex\right\vert}}
\newcommand{\abs}[1]{\left\lvert#1\right\rvert}
\newcommand{\norm}[1]{\left\lVert#1\right\rVert}
\newcommand{\enbrace}[1]{\left\lbrace#1\right\rbrace}
\newcommand{\enpar}[1]{\left(#1\right)}
\newcommand{\Gt}{\ensuremath{\mathcal{G}}}
\newcommand{\EB}{\ensuremath{\mathcal{E}}}
\newcommand{\Bt}{\ensuremath{\mathcal{B}}}
\newcommand{\It}{\ensuremath{\mathcal{I}}}
\newcommand{\XB}{\ensuremath{\mathcal{X}}}
\newcommand{\YB}{\ensuremath{\mathcal{Y}}}

\newcommand{\Ind}{\ensuremath{\mathbbm{1}}}
\newcommand{\Sym}{\ensuremath{\mathbb{S}}}
\newcommand{\SymG}{\ensuremath{\mathrm{Sym}}}
\newcommand{\NN}{\ensuremath{\mathbb{N}}}
\newcommand{\EE}{\ensuremath{\mathbb{E}}}
\newcommand{\DD}{\ensuremath{\mathbb{D}}}
\newcommand{\FF}{\ensuremath{\mathbb{F}}}
\newcommand{\RR}{\ensuremath{\mathbb{R}}}
\newcommand{\XX}{\ensuremath{\mathbb{X}}}
\newcommand{\YY}{\ensuremath{\mathbb{Y}}}
\newcommand{\ZZ}{\ensuremath{\mathbb{Z}}}
\newcommand{\BB}{\ensuremath{\mathbb{B}}}
\newcommand{\ee}{\ensuremath{\bm{e}}}
\newcommand{\xx}{\ensuremath{\bm{x}}}
\newcommand{\yy}{\ensuremath{\bm{y}}}
\newcommand{\Nr}{\ensuremath{\bm{N}}}
\newcommand{\tb}{\ensuremath{\bm{\tau}}}
\newcommand{\dist}{\ensuremath{\bm{\delta}}}
\newcommand{\prim}{\ensuremath{\bm{\sigma}}}
\newcommand{\FuFu}{\ensuremath{\bm{\Lambda}}}
\newcommand{\usdf}{\ensuremath{\bm{\varphi}}}
\newcommand{\ww}{\ensuremath{\bm{w}}}

\newcommand{\uu}{\ensuremath{\bm{u}}}
\newcommand{\unc}{\ensuremath{\bm{k}}}
\newcommand{\uunc}{\ensuremath{\widetilde{\bm{k}}}}
\DeclareMathOperator{\spn}{span}
\DeclareMathOperator{\supp}{supp}

\begin{document}
\title[Conditional bases with Property~(A)]{Conditional bases with Property~(A)}
\author[Albiac]{Fernando Albiac}
\address{Department of Mathematics, Statistics, and Computer Sciencies--InaMat2\\
Universidad P\'ublica de Navarra\\
Campus de Arrosad\'{i}a\\
Pamplona\\
31006 Spain}
\email{fernando.albiac@unavarra.es}
\author[Ansorena]{Jos\'e L. Ansorena}
\address{Department of Mathematics and Computer Sciences\\
Universidad de La Rioja\\
Logro\~no\\
26004 Spain}
\email{joseluis.ansorena@unirioja.es}
\author[Berasategui]{Miguel Berasategui}
\address{Miguel Berasategui\\
UBA - Pab I, Facultad de Ciencias Exactas y Naturales\\
Universidad de Buenos Aires\\
(1428), Buenos Aires, Argentina}
\email{mberasategui@dm.uba.ar}
\author[Bern\'a]{Pablo M. Bern\'a}
\address{Pablo M. Bern\'a\\
Departamento de Métodos Cuantitativos, CUNEF Universidad\\
Madrid, 28040 Spain}
\email{pablo.berna@cunef.edu}
\begin{abstract}
Property~(A) is a week symmetry condition that plays a fundamental role in the characterization of greedy-type bases in the \emph{isometric} case, i.e., when the constants involved in the study of the efficiency of the thresholding greedy algorithm in Banach spaces are sharp. In this note we build examples of Banach spaces with Schauder bases that have Property~(A) but fail to be unconditional, thus settling a long standing problem in the area. As a by-product of our work we hone our construction to produce counterexamples that solve other open questions in the isometric theory of greedy bases.
\end{abstract}
\thanks{F.\ Albiac and J.\ L.\ Ansorena acknowledge the support of the Spanish Ministry for Science and Innovation under Grant PID2022-138342NB-I00 for \emph{Functional Analysis Techniques in Approximation Theory and Applications (TAFPAA)}}

\thanks{M. Berasategui is supported by the Grants CONICET PIP
11220200101609CO y ANPCyT PICT 2018-04104 (Consejo Nacional de Investigaciones Científicas y Técnicas y Agencia Nacional de Promoción de la Investigación, el Desarrollo Tecnológico y la Innovación, Argentina)}

\thanks{Pablo M. Berná is supported by the Grant PID2022-142202NB-I00 (Agencia Estatal de Investigación, Spain)}
\maketitle
\section{Introduction and background}\noindent
Arguably, the isometric theory of greedy bases was initiated twenty years ago in \cite{AW2006} with the characterization of those bases which provide optimal performance of the thresholding greedy algorithm in a given Banach space where nonlinear approximation using bases is taking place.

Let $\XB=(\xx_n)_{n=1}^{\infty}$ be a (Schauder) basis in a Banach space $\XX$ and for $x\in \XX$ and $m\in \NN$ let $\sigma_{m}(x)$ be the (theoretical) error of the best nonlinear $m$-term approximation to $x$, that is, $\sigma_{m}(x)=\inf\norm{ x-y}$, where the infimum is taken over all linear combinations $y=\sum_{n} c_{n}\,\xx_{n}$ with at most $m$ nonzero coefficients $c_{n}$. By definition, for a given $x\in \XX$ and $m\in \NN$, an $m$-term greedy approximation of $x$, $\Gt_{m}(x)$, is obtained by retaining $m$ terms with the largest $\abs{a_{n}}$'s from the series expansion $x=\sum_{n=1}^{\infty}a_{n}\, \xx_{n}$ and replacing the remaining $a_{n}$'s by zeros. The sequence of non-linear operators $(\Gt_m)_{m=1}^\infty$ is called the thresholding greedy algorithm (TGA for short) relative to $\XB$. The basis $\XB$ is called greedy with constant $C$ or $C$-greedy ($C\ge 1$) if
\[
\norm{ x-\Gt_m (x)}\Vert\le C\sigma_{m}(x)
\]
for all $x\in \XX$ and all $m\in \NN$. These bases were formally introduced by Konyagin and Temlyakov in their pioneering work \cite{KoTe1999}, where they were intrinsically characterized as unconditional basis with the additional property of being democratic. We are interested in the quantitative version of this result, which reads as follows.

\begin{theoremintro}[\cite{KoTe1999}*{Theorem 1}]
Let $\XB$ be a basis of a Banach space $\XX$. If $\XB$ is $C$-greedy then $\XB$ is unconditional with suppression unconditionality constant $C$ and democratic with democratic constant $C$. Conversely, if $\XB$ is unconditional with suppression unconditionality constant $K_{s}$, lattice unconditionality constant $K_u$, and democratic with constant $\Delta$, then $\XB$ is greedy with greedy constant $K_s+K_s K_u^2 \Delta$.
\end{theoremintro}

Since various approximation algorithms converge trivially when some appropriate constant is one, the isometric case is of special interest. For instance, it is easy to see that any $1$-symmetric basis of a Banach space, in particular any orthonormal basis of a separable Hilbert space, is $1$-greedy. Conversely, any $1$-greedy basis $(\xx_{n})_{n=1}^{\infty}$ of a Hilbert space is orthogonal and $\norm{\xx_n}= \norm{\xx_m}\Vert$ for all $n$, $m\in \NN$. In the context of Banach spaces, Theorem~A immediately yields that a $1$-greedy basis has both suppression unconditionality constant and democratic constant equal to $1$. However, this is not a characterization of bases with greedy constant $1$. Answering a question raised by Wojtaszczyk (see \cite{Wojt2003}*{Problem 1}), the authors found a satisfactory description of $1$-greedy bases in the general case.

\begin{theoremintro}[\cite{AW2006}*{Theorem 3.4}]
A basis $\XB =(\xx_{n})_{n=1}^{\infty}$ of a Banach space $\XX$ is $1$-greedy if and only if it is $1$-suppression unconditional and satisfies Property~(A).
\end{theoremintro}

Loosely speaking, Property~(A) allows the largest coefficients of a vector to be relocated and rotated without altering its norm. In this regard, Property~(A) can be seen as a weak form of symmetry which implies not only $1$-democracy but also $1$-superdemocracy. To define it properly, we say that $y=\sum_{n}b_{n}\, \xx_{n}\in \XX$ is a \emph{greedy rearrangement} of $x=\sum_{n}\, a_{n} \, \xx_{n}\in\XX$ (with respect to the basis $\XB=(\xx_{n})_{n=1}^{\infty}$ of $\XX$) if there exists a bijection $\pi\colon \NN \to \NN$ such that $|b_{\pi(n)}|=|a_{n}|$ for all $n\in\NN$, and $\pi(k) = k$ and $a_{k} = b_{k}$ whenever $\abs{a_k}<\sup_n \abs{a_n}$. The basis $\XB$ satisfies Property~(A) if $\norm{x} =\norm{y}$ whenever $y$ is a greedy rearrangement of $x$.

It is worth it emphasizing that, despite the fact that they are closely related properties, $1$-superdemocracy cannot supplant Property~(A) in determining the isometric properties of greedy-like bases, even when combined with other stronger features of the basis. Indeed, the canonical basis of Garling sequence space is $1$-subsymmetric (hence $1$-lattice unconditional and $1$-superdemocratic) but fails to be $1$-greedy because it does not have Property~(A)! (see \cite{AW2006}*{Example 5.4}).

The neat description of $1$-greedy bases provided by Theorem~B motivated further work in the isometric theory of greedy bases. The next movers in this direction were Dilworth et al., who in \ \cite{DOSZ2011} solved most of the questions that were raised in \cite{AW2006}. For instance, they found examples of $1$-subsymmetric bases that fail to be symmetric while being $1$-greedy, and examples of bases that are $1$-greedy but not subsymmetric. Later on it was shown in \cite{AAW2018b} that Garling sequence spaces can be renormed so that its canonical basis is $1$-greedy.

Theorem~B also inspired the analysis of other new breeds of bases that arise naturally in the study of the TGA. This study led to the following characterizations of the other two important types of greedy-like bases, namely quasi-greedy, and almost greedy bases in the isometric case, precisely in terms of the same two ingredients that we find in the characterization of $1$-greedy bases but in disjoint occurrences.

\begin{theoremintro}[\cite{AlbiacAnsorena2016c}*{Theorem 2.1}]
A basis of a Banach space is $1$-quasi-greedy if and only if it is 1-suppression unconditional.
\end{theoremintro}

\begin{theoremintro}[\cite{AlbiacAnsorena2017b}*{Theorem 2.3}]
A basis of a Banach space is $1$-almost greedy if and only if it satisfies Property~(A).
\end{theoremintro}

Since almost greedy basis are in particular quasi-greedy, one could expect that when the almost greedy constant is sharp, the implication would still hold, i.e., that being $1$-almost greedy implies being $1$-quasi-greedy. In light of Theorems~C and D, we arrive naturally at the following question, which enquires about the overlapping of the two properties that characterize $1$-greedy bases.

\begin{questionintro}
Does Property~(A) imply unconditionality with suppression unconditionality constant $K_s = 1$?
\end{questionintro}

If we restrict our attention to unconditional bases, the study of the relation between Property~(A) and unconditionality reduces to the problem of determining whether Property~(A) implies some upper bound for the unconditionality constant of the basis.

\begin{questionintro}
Does an unconditional basis with Property~(A) always have $K_s = 1$?
\end{questionintro}

Theorem~C is somewhat surprising since it connects nonlinear properties in approximation theory in Banach spaces (quasi-greediness) with linear properties (such as unconditionality). Besides, it exhibits how an isometric (hence, \emph{quantitative})
property could lead to an improvement of the isomorphic (i.e., \emph{qualitative}) behaviour of a basis. It is therefore natural to ask whether this is also the case with almost greedy bases.

\begin{questionintro} Is there a conditional basis with Property~(A)?
\end{questionintro}

Both Questions E and F were solved in the negative in the recent article \cite{AABHO2024}.

\begin{theoremintro}[\cite{AABHO2024}*{Propositions 2.4 and 3.4}]
For each $K<3/2$ there is a distortion of $\ell_1$ so that the standard unit vector basis still satisfies Property~(A) and is unconditional with suppression unconditional constant $K_{s}\ge K$.
\end{theoremintro}

However, Question~G (which was already implicit in \cite{AW2006} and has been explicitly raised in \cite{AlbiacAnsorena2017b}*{Problem 4.4}, \cite{AABHO2024}*{Question 1.5}, and \cite{AAT2024}*{Problem 6}) has remained elusive for the last twenty years. In this note we answer Question~G negatively, thus closing the cycle of ideas revolving around isometric greedy-like bases. We will split in two sections the construction of the Banach spaces and the bases that will serve as counterexamples. In Section~\ref{Sect2} we will introduce the necessary ingredients for the construction and state some auxiliary results; Section~\ref{sect:main} is devoted to defining the basis and to exhibiting its main properties. In Section~\ref{SectNew} we will apply our construction to find counterexamples that settle other problems in the theory. On the one hand, answering Question 3.10 from \cite{AABHO2024}, we show the existence of an unconditional Schauder basis that has Property~(A) (together with other isometric-like features) but whose suppression unconditional basis constant can be as large as we wish; on the other hand we also solve Problem 3.4 from \cite{AlbiacAnsorena2017b} by building a $1$-suppression quasi-greedy basis which fails to be unconditional. This is somewhat surprising because $1$-quasi-greedy bases happen to be unconditional with suppression unconditional constant equal to 1. In Section~\ref{Sect4} we make a brief analysis of the geometry of the spaces that witness our construction and we close in Section~\ref{OQ} with a few hints on how to take further the research on this topic.

Throughout this note we will use the notation and terminology that are by now standard in greedy approximation
and Banach space theory, as can be found in \cite{AABW2021} and \cite{AlbiacKalton2016}.
We will also use concepts from function space theory, for which we refer the reader to \cite{BennettSharpley1988}.
For the sake of expositional easy, other more specific terminology will be introduced in context when it is used.

\section{The ingredients of the construction}\label{Sect2}\noindent
To show the existence of conditional bases with Property~(A) we will build the Banach spaces where we can find them with the aid of two ingredients. The first ingredient will be a \emph{weight}, i.e., a sequence $\ww=(w_n)_{n=1}^\infty$ in $[0,\infty)$ with $w_1>0$. We will use $\prim[\ww]=(s_m)_{m=1}^\infty$ for the \emph{primitive sequence} of $\ww$ given by
\[
s_m=\sum_{n=1}^m w_n, \quad m\in \NN.
\]
We say that $\ww$ is null if $\lim_n w_n=0$, and that it is summable if $\sum_{n=1}^\infty w_n<\infty$.

The other ingredient of our construction will be a \emph{symmetric Banach space}, i.e., a Banach space $\Sym$ built from a rearrangement invariant function norm $\norm{\cdot}_{\Sym}$ over $\NN$ such that $c_{00} \subset \Sym\subset\FF^\NN$, where $\FF$ denotes the real or complex field $\FF$, and $c_{00}$ denotes the linear span of the unit vectors $(\ee_n)_{n=1}^\infty$ of $\FF^\NN$. More generally, spaces constructed from a rearrangement invariant function quasi-norm will be called \emph{symmetric quasi-Banach spaces}.

Let $\EE$ (resp., $\DD$) denote the set consisting of all scalars of modulus one (resp., modulus at most one). Let $(\ee_n^*)_{n=1}^\infty$ be the sequence of dual functionals defined by $\ee_n^*(f)=a_n$ for all $n\in\NN$ and all $f=(a_k)_{k=1}^\infty\in\FF^\NN$. Given $A\subset\NN$ finite and $\varepsilon=\enpar{\varepsilon_n}_{n\in A}$ in $\EE$, we put
\[
\Ind_{\varepsilon,A}=\sum_{n\in A} \varepsilon_n \, \ee_n, \quad\text{and}\quad
\Ind^*_{\varepsilon,A}=\sum_{n\in A} \varepsilon_n \, \ee_n^*.
\]
Associated with a symmetric quasi-Banach space $\Sym$ there is a sequence $\FuFu[\Sym]=(\Lambda_m)_{m=1}^\infty$ in $(0,\infty)$, called the \emph{fundamental function} of $\Sym$, such that
\[
\norm{\Ind_{\varepsilon,A}}_\Sym=\Lambda_{\abs{A}}, \quad A\subset\NN, \, \abs{A}<\infty,\, \varepsilon\in \EE^A.
\]
More generally, given a basis $(\xx_n)_{n=1}^\infty$ in a Banach space $\XX$, $A\subset\NN$ with $\abs{A}<\infty$, and $\varepsilon\in\EE^A$ we put
\[
\Ind_{\varepsilon,A}[\XB,\XX]=\sum_{n\in A} \varepsilon_n \, \xx_n,
\]
and define the \emph{fundamental function} of $\XB$ as
\[
\usdf_m[\XB,\XX]=\sup\enbrace{ \norm{\Ind_{\varepsilon,A}[\XB,\XX]} \colon \abs{A}\le m, \, \varepsilon\in\EE^A}, \quad m\in\NN.
\]

Now, given a weight $\ww$ and a symmetric Banach space $\Sym$ we define
\[
\norm{\cdot}_{\Sym[\ww]} \colon \FF^\NN \to [0,\infty], \quad f\mapsto \norm{f}_{\Sym[\ww]}=\sup_{\abs{A}<\infty} \frac{s_{\abs{A}}}{\Lambda_{\abs{A}}}\norm{S_A(f)}_\Sym,
\]
where $S_A\colon \FF^\NN\to \FF^\NN$ is the canonical coordinate projection onto the linear space of all sequences supported on $A$.

Let $\Sym[\ww]$ be the space of all sequences $f\in\FF^\NN$ such that $ \norm{f}_{\Sym[\ww]}<\infty$.
If $f^*=(a_n)_{n=1}^\infty$ is the non-increasing rearrangement of (the absolute value of) a sequence $f\in c_{00}$, then
\[
\norm{f}_{\Sym[\ww]}=\sup_{m\in\NN} \frac{s_m}{\Lambda_m} \norm{\sum_{n=1}^m a_n \, \ee_n}_\Sym.
\]
In particular, given $A\subset\NN$ finite and $\varepsilon\in\EE^A$,
\begin{equation}\label{eq:FFSw}
\norm{\Ind_{\varepsilon,A}}_{\Sym[\ww]}= \Lambda_{\abs{A}} \sup_{m\ge \abs{A}} \frac{s_m}{\Lambda_m}.
\end{equation}
Thus, the space $\Sym[\ww]$ associated with $\Sym$ and $\ww$ is a symmetric Banach space if and only if $(s_n/\Lambda_n)_{n=1}^\infty$ is bounded; otherwise $\Sym[\ww]$ is a trivial vector space. If $\Sym$ is a Banach space, so is $\Sym[\ww]$.

We will need to compare in some way $\Sym[\ww]$ with Lorentz sequence spaces, whose definition we recall for the sake of self-reference.

Given a weight $\ww$ and $0<q\le\infty$ there is a unique map $\norm{\cdot }_{q,\ww}\colon\FF^\NN \to[0,\infty]$ such that
\[
\norm{f}_{q,\ww}=\enpar{\sum_{n=1}^\infty (a_n s_n)^q \frac{w_n}{s_n}}^{1/q}, \quad f\in c_{00}, \, f^*=(a_n)_{n=1}^\infty,
\]
with the usual modification if $q=\infty$, and
\[
\norm{f}_{q,\ww}=\sup_{\abs{A}<\infty} \norm{S_A(f)}_{q,\ww}, \quad f\in\FF^\NN.
\]
The Lorentz space $d_{q}(\ww)$ is the space of all sequences $f\in\FF^\NN$ such that $\norm{f}_{q,\ww}<\infty$. If $0<p<\infty$ and $\ww$ is equivalent to $(n^{1/p})_{n=1}^\infty$, the space $d_q(\ww)$ is the classical Lorentz sequence space $\ell_{p,q}$.

As is customary, given families of nonnegative numbers $(\alpha_i)_{i\in I}$ and $(\beta_i)_{i\in I}$, the symbol $\alpha_i\lesssim \beta_i$ for $i\in I$ means that there is a constant $C$ such that $\alpha_i \le C \beta_i$ for all $i\in I$. If $\alpha_i\lesssim \beta_i$ and $\beta_i\lesssim \alpha_i$ for $i\in I$ we say that $(\alpha_i)_{i\in I}$ and $(\beta_i)_{i\in I}$ are equivalent, and we write $\alpha_i\approx \beta_i$ for $i\in I$.

It is known \cite{AABW2021} that if $0<q<\infty$ and $\uu=(u_n)_{n=1}^\infty$ are such that $\prim[\uu]\approx \prim^q[\ww]$ then
\begin{equation}\label{eq:EquivalentLorentz}
\norm{f}_{q,\ww}\approx \norm{f}_{(q,\uu)} :=\enpar{\sum_{n=1}^\infty a_n^q u_n}^{1/q},
\quad f\in c_{00}, \, f^*=(a_n)_{n=1}^\infty.
\end{equation}

The gauge $\norm{\cdot}_{q,\ww}$ is not a function norm in general. Notwithstanding, if the primitive weight of $\ww$ is doubling then $\norm{\cdot}_{q,\ww}$ is a rearrangement function quasi-norm, whence $d_q(\ww)$ is a symmetric quasi-Banach space.

Given $0<q< p \le \infty$ we have
\[
\norm{ \cdot}_{p,\ww} \lesssim \norm{ \cdot}_{q,\ww},
\]
and unless $\ww$ is summable, in which case $d_q(\ww)=\ell_\infty$ for all $q\in(0,\infty]$,
$\norm{ \cdot}_{p,\ww}$ and $\norm{ \cdot}_{q,\ww}$ are not equivalent. However, we cannot distinguish the spaces using constant coefficient vectors; in fact, $\FuFu[d_q(\ww)]\approx \prim[\ww]$ for all $0<q\le \infty$.

If $\ww$ is non-increasing, $\norm{\cdot}_{1,\ww}$ is a function norm. More generally, we have the following.
\begin{lemma}\label{lem:LorentzNormable}
Let $1\le q<\infty$ and let $\uu=(u_n)_{n=1}^\infty$ be a non-increasing weight. Then the gauge $\norm{\cdot}_{(q,\uu)}$ defined in \eqref{eq:EquivalentLorentz} extends to a function norm.
\end{lemma}
\begin{proof}
By the rearrangement inequality,
\[
\norm{f}_{(q,\uu)}=\sup_{\pi\in \SymG(\NN)}\enpar{ \sum_{n=1}^\infty \abs{b_{\pi(n)}}^q u_n }^{1/q}, \quad f=(b_n)_{n=1}^\infty\in c_{00}.
\]
Therefore $\norm{\cdot}_{(q,\uu)}$ is subadditive.
\end{proof}

A sequence $(t_m)_{m=1}^\infty$ of nonnegative scalars is said to be \emph{essentially decreasing} if $t_n\lesssim t_m$ for all $n$, $m\in\NN$ such that $n\ge m$. The sequence $(t_m)_{m=1}^\infty$ is essentially decreasing if and only if there is a non-increasing sequence $(t_m')_{m=1}^\infty$ such that $t_m \approx t_m'$ for $m\in \NN$.

\begin{lemma}\label{lem:embedding}
Let $\ww$ be a weight and $\Sym$ be a symmetric Banach space. Set $\prim[\ww]=(s_m)_{m=1}^\infty$, $\FuFu[\Sym]=(\Lambda_m)_{m=1}^\infty$ and $r_m=s_m/\Lambda_m$, $m\in\NN$. Suppose that $(r_m)_{m=1}^\infty$ is essentially decreasing and put $C=\sup_{m\ge n} r_m/r_n$. Then
\[
\norm{f}_{\infty,\ww} \le \norm{f}_{\Sym[\ww]} \le C\norm{f}_{1,\ww},\quad f\in\FF^\NN,
\]
so that $d_1(\ww) \subset \Sym[\ww]\subset d_\infty(\ww)$ continuously. Besides,
\begin{enumerate}[label=(\roman*), leftmargin=*, widest=ii]
\item\label{it:embedding:A} If $r_m\approx 1$ for $m\in\NN$, then $\Sym[\ww]$ is just a renorming of $\Sym$.
\item\label{it:embedding:B} If $\lim_m r_m=0$, then the unit vector system of $\Sym[\ww]$ fails to be boundedly complete.
\item\label{it:embedding:C} If $\lim_m r_m=0$ and $\ww$ is non-summable, then $\Sym[\ww]\not=d_1(\ww)$.
\end{enumerate}
\end{lemma}

\begin{proof}
By \eqref{eq:FFSw}, the fundamental function of $\Sym[\ww]$ is equivalent to the primitive weight of $\ww$. This fact implies the desired embeddings for a wide class of sequence spaces containing symmetric Banach spaces (see \cite{AAB2021}*{Theorem 9.12} or \cite{AADK2016}*{Lemma 6.1} for a more elemental proof that works for unconditional democratic bases).

\ref{it:embedding:A} is clear. To prove \ref{it:embedding:B} we will construct a block basis sequence $(\xx_j)_{j=1}^\infty$ equivalent to the unit vector system of $\ell_\infty$. Choose $(n_k)_{k=1}^\infty$ such that
\[
\sum_{j=1}^{k-1} \frac{1}{r_{n_j}} \le \frac{1}{r_{n_k}}, \quad k\in\NN.
\]
For $k\in\NN_0$ let $m_k=\sum_{j=1}^k n_j$, and define
\[
\xx_j=\frac{1}{s_{n_j}} \Ind_{\sigma_j}, \quad \sigma_j=(m_{j-1},m_j]\cap \ZZ.
\]
We shall prove that $g:=\sum_{j=1}^\infty \xx_j\in \Sym[\ww]$. Note that $g=(a_n)_{n=1}^\infty$ is nonnegative and non-increasing. For each $k\in\NN$ we have
\[
\norm{\sum_{n=1}^{m_k} a_n \, \ee_n}_\Sym
=\norm{\sum_{j=1}^k \xx_j}_\Sym
\le \sum_{j=1}^k \norm{\xx_j}_\Sym
=\sum_{j=1}^k \frac{1}{r_{n_j}}\le \frac{2}{r_{n_k}}.
\]
Pick $m\in\NN$ and let $k\in\NN$ be such that $m_{k-1}<m\le m_k$. Then, with the convention that $n_0=0$ and $r_0=\infty$,
\begin{multline*}
\norm{\sum_{n=1}^m a_n \, \ee_n}_\Sym
\le \norm{\sum_{n=1}^{m_{k-1}} a_n \, \ee_n}_\Sym
+ \norm{\sum_{n=1+m_{k-1}}^m a_n \, \ee_n}_\Sym\\
\le \frac{2}{r_{n_{k-1}}} + \frac{\Lambda_{m-m_{k-1}}}{s_{m_k}}
\le \frac{2C}{r_m} +\frac{\Lambda_m}{s_m}=(1+2C) \frac{\Lambda_m}{s_m}.
\end{multline*}
Taking into account that the unit vector system of $d_1(\ww)$ is boundedly complete whenever $\ww$ is non-summable,
\ref{it:embedding:C} is a consequence of \ref{it:embedding:B}.
\end{proof}

We record a few well-known properties of symmetric Banach spaces to which we will refer below.

\begin{proposition}\label{prop:dualsymmetric}Given a symmetric Banach space $\Sym$, let us denote by $\Sym'$ its associate space and by $\Sym_0$ the closure of $c_{00}$ in $\Sym$. Let $(\Lambda_n)_{n=1}^\infty$ and $(\Lambda_n')_{n=1}^\infty$ be the fundamental functions of $\Sym$ and $\Sym'$, respectively. The following statements hold.
\begin{enumerate}[label=(\roman*), leftmargin=*, widest=iii]
\item\label{it:sym:a} $\Sym_0^*=\Sym'$.
\item\label{it:sym:b} $\Lambda_n \, \Lambda'_n=n$ for all $n\in\NN$.
\item\label{it:sym:c} If $\Lambda_n\approx 1$ for $n\in\NN$, then $\Sym=\ell_\infty$.
\item\label{it:sym:d} If $\Lambda_n\approx n$ for $n\in\NN$, then $\Sym=\ell_1$.
\end{enumerate}
\end{proposition}

\begin{proof}
\ref{it:sym:a} and \ref{it:sym:b} are particular cases of \cite{BennettSharpley1988}*{Theorem 4.1 and
Theorem5.2} (see also \cite{LinTza1977}*{Proposition 3.a.6}). \ref{it:sym:c} is clear, and \ref{it:sym:d} follows by duality from \ref{it:sym:b} and \ref{it:sym:c}.
\end{proof}

\begin{corollary}\label{cor:Uniquel1loo}
Let $\ww$ be a non-increasing weight. Suppose $\Sym$ is a symmetric Banach space with
$\FuFu[\Sym] \approx \prim[\ww]$. If $\Sym \not=d_1(\ww)$ then $\ww$ is null and non-summable.
\end{corollary}
\begin{proof}
If $\ww$ is not null then $d_1(\ww)=\ell_1$. If $\ww$ is summable then $d_1(\ww)=\ell_\infty$. Hence the result follows from Proposition~\ref{prop:dualsymmetric}.
\end{proof}

Not all pairs $(\ww,\Sym)$ formed by a weight $\ww$ and a symmetric Banach space $\Sym$ will be suited for our construction.

\begin{definition}
We will say that a symmetric Banach space $\Sym$ \emph{couples with} a non-increasing, non-summable weight $\ww$ if there are a constant $c\in(0,\infty)$ and a non-increasing sequence $\tb=(t_m)_{m=1}^\infty$ equivalent to $(s_n/\Lambda_n)_{n=1}^\infty$ such that, for all $m\in\NN_0$ and all $m$-tuples $(a_n)_{n=1}^m$ in $\DD$,
\begin{equation}\label{eq:KPA}
t_m \norm{ \sum_{n=1}^m a_n \, \ee_n }_{\Sym} +c w_{m+1} \le t_{m+1} \norm{ \enpar{\sum_{n=1}^m a_n \, \ee_n} + \ee_{m+1}}_{\Sym}.
\end{equation}
In this case we say that the pair $(\tb,c)$ is \emph{congruent with} the pair $(\ww,\Sym)$. If, besides, $\Sym \not= d_1(\ww)$, we say that $\Sym$ \emph{properly couples} with $\ww$.
\end{definition}

To make sense of the equation \eqref{eq:KPA} for $m=0$ we adopt the convention that $t_0=1$ and that a sum over an empty set is zero.

\begin{lemma}\label{lem:Swnotd1w}
Suppose that a symmetric Banach space $\Sym$ couples with a non-increasing, non-summable weight $\ww$. Then,
\[
d_1(\ww)\subset \Sym[\ww]\subset d_\infty(\ww).
\]
Moreover, if $\Sym$ properly couples with $\ww$ then $ d_1(\ww) \subsetneq \Sym[\ww]$, and $\ww$ is null.
\end{lemma}

\begin{proof}
The sequence $(s_n/\Lambda_n)_{n=1}^\infty$ is essentially decreasing, so the desired embeddings follow from Lemma~\ref{lem:embedding}. If $\Sym$ properly couples with $\ww$, combining Corollary~\ref{cor:Uniquel1loo} with the `besides' part of Lemma~\ref{lem:embedding}, gives the desired results.
\end{proof}

Before building
the advertised conditional bases with Property~(A) let us make sure there exist pairs $(\ww,\Sym)$ so that $\Sym$ properly couples with $\ww$.

Given a weight $\ww$ and $0<q<\infty$, we denote by $\ww_q$ the weight defined by $\prim[\ww_q]=\prim^q[\ww]$.

\begin{lemma}\label{lem:GeneralCoupling}
Let $\ww$ be a non-increasing and non-summable weight. Suppose $1\le q <\infty$ and let $\uu=(u_n)_{n=1}^\infty$ be another non-increasing weight. Put $\prim[\ww]=(s_m)_{m=1}^\infty$, $\prim[\uu]=(\lambda_m)_{m=1}^\infty$, and $t_m=s_m\lambda_m^{-1/q}$ for all $m\in\NN$. Assume that $\tb=(t_m)_{m=1}^\infty$ is non-increasing. Then $\Sym:=d_q(\uu_{1/q})$ endowed with the norm $\norm{\cdot}_{(q,\uu)}$ defined in \eqref{eq:EquivalentLorentz} is a symmetric Banach space that couples with $\ww$. In fact, $(\tb,1)$ is congruent with $(\ww,\Sym)$. Moreover, if $q>1$ or $\lim_m t_m=0$, then $\Sym$ properly couples with $\ww$.
\end{lemma}

\begin{proof}
$\Sym$ is a symmetic Banach space by Lemma~\ref{lem:LorentzNormable}. Note that $\FuFu[\Sym]=(\lambda_m^{1/q})_{m=1}^\infty$. Given $0\le A\le B<\infty$, $0\le b\le a <\infty$, and $0\le \beta<\alpha<\infty$, the map
\[
x \mapsto \beta \enpar{B+b x^q}^{1/q}- \alpha \enpar{A+a x^q}^{1/q}, \quad 0\le x<\infty,
\]
is non-increasing. Therefore, given $m\in\NN$, the function $F$ given by
\[
F\enpar{\bm{a}}=t_{m+1}\norm{\enpar{\sum_{n=1}^m a_n \, \ee_n}+\ee_{m+1}}-t_m\norm{\sum_{n=1}^m a_n \, \ee_n}, \quad \bm{a}=(a_n)_{n=1}^m,
\]
attains its minimum value on $\DD^m$ at the $m$-tuple $\bm{b}=(b_n)_{n=1}^m$ given by
$b_n=1$ for all $n=1$, \dots, $m$. Since $F\enpar{\bm{b}}=w_{m+1}$, $\Sym$ couples with $\ww$. The `moreover' part is clear.
\end{proof}

Before we go on we need some terminology.

A sequence $(s_n)_{n=1}^\infty$ of positive numbers is said to have the \emph{lower regularity property} (LRP for short) if there is $r\in\NN$ such that
\[
s_{rn}\ge 2 s_n, \quad n\in\NN.
\]
A sequence $\prim=(s_n)_{n=1}^\infty$ is said to have the \textit{upper regularity property} (URP for short) if
its \emph{dual sequence} $\prim^*=(n/s_n)_{n=1}^\infty$ has the LRP, i.e.,
there is $r\in\NN$ such that
\[
s_{rn}\le \frac{1}{2} r s_n, \quad n\in\NN.
\]

Given $\alpha\in\RR$, the power sequence $(n^\alpha)_{n=1}^\infty$ has the LRP if and only if $\alpha>0$, and
has the URP if and only if $\alpha<1$.

Let $\Sym_1$ and $\Sym_2$ be symmetric quasi-Banach spaces with $\Sym_1\subset\Sym_2$. We have that $\Sym_1=\Sym_2$ if and only
\[
\norm{f}_{\Sym_1} \lesssim \norm{f}_{\Sym_2}, \quad f\in c_{00}.
\]
Following \cite{AAB2021}, in order to have a quantitave estimate of how close the spaces $\Sym_1$ and $\Sym_2$ are, for each $m\in\NN$ we define $\dist_m[\Sym_1,\Sym_2]$ as the optimal constant $C$ such that
\[
\norm{\sum_{n=1}^m a_n \, \ee_n}_{\Sym_1} \le C \norm{\sum_{n=1}^m a_n \, \ee_n}_{\Sym_2}
\]
for all $m$-tuples $(a_n)_{n=1}^m$ in $\FF$.

Given a weight $\ww$ with primitive sequence $(s_n)_{n=1}^\infty$, the space $\ell_1[\ww]$ constructed from the symmetric Banach space $\Sym=\ell_1$ is the \emph{Marcinkiewicz space} $m(\ww)$ consisting of all sequences $(a_n)_{n=1}^\infty$ such that
\[
\sup_{\abs{A}<\infty}\frac{s_{\abs{A}}}{\abs{A}} \sum_{n\in A} \abs{a_n}<\infty.
\]

We record some properties of Lorentz and Marcinkiewicz sequence spaces for further reference.
\begin{proposition}[see \cites{ABW2023,AAB2021}]\label{lem:ML}
Assume that a weight $\ww=(w_n)_{n=1}^\infty$ is non-increasing, null, and non-summable. Set $\prim[\ww]=(s_n)_{n=1}^\infty$ and
\[
H_m[\ww]:=\sum_{n=1}^m \frac{w_n}{s_n}, \quad m\in\NN.
\]
\begin{enumerate}[label=(\roman*), leftmargin=*, widest=iii]
\item If $\prim[\ww]$ has the URP, then $d_\infty(\ww)=m(\ww)$.
\item Given $0<p<q\le\infty$ $d_p(\ww)\subset d_q(\ww)$ and
\[
\dist_m(d_p(\ww), d_q(\ww))=\enpar{H_m[\ww]}^{1/p-1/q}, \quad m\in\NN.
\]
\item If $\prim[\ww]$ has the URP, then $H_m[\ww] \approx \log(1+m)$ for $m\in\NN$.
\item\label{it:ML:6} $(d_1(\ww))^*=m(\ww)$.
\end{enumerate}
\end{proposition}

\begin{proposition}\label{lem:L1coupleswith}
The space $\ell_1$ couples with any non-increasing and non-summable weight $\ww$. In fact, if $\prim[\ww]=(s_m)_{m=1}^\infty$ and
\[
\tb=(s_m/m)_{m=1}^\infty,
\]
then $(\tb,1)$ is congruent with $(\ww,\ell_1)$. If $\ww$ is null, then $\ell_1$ properly couples with $\ww$. Moreover, if $\prim[\ww]$ has the URP, then
\[
\dist(d_1(\ww),\ell_1[\ww])\approx \log(1+m), \quad m\in\NN.
\]
\end{proposition}

\begin{proof}
Since $\tb$ is non-increasing, $\ell_1$ couples with $\ww$ by Lemma~\ref{lem:GeneralCoupling}. If $\prim[\ww]$ has the URP then, by Proposition~\ref{lem:ML},
\[
\dist_m(d_1(\ww),\ell_1[\ww]) \approx \dist_m(d_1(\ww),d_\infty(\ww))\approx \log(1+m), \quad m\in\NN.\qedhere
\]
\end{proof}
Let $\ww$ be a weight and assume as usual that it is non-increasing, null, and non-summable. Lemma~\ref{lem:embedding}\ref{it:embedding:A} leads us to wonder about the existence of symmetric Banach spaces $\Sym$ with $\FuFu[\Sym]\approx \prim[\ww]$ that properly couple with $\ww$. If $\prim[\ww]$ is equivalent to $(n^\alpha)_{n=1}^\infty$ for some $0<\alpha<1$ we give a positive answer to this question.

\begin{proposition}\label{lemma:lp}
Let $1<p<\infty$ and $\ww=(w_n)_{n=1}^\infty$ be a non-increasing weight equivalent to $(n^{1/p-1})_{n=1}^\infty$. Then $\ell_p$ properly couples with $\ww$. In fact, if
\[
c=\inf_n\frac{n^{1/p}-(n-1)^{1/p}}{w_n}
\]
and $\tb=(1)_{m=1}^\infty$, then $(\tb,c)$ is congruent with $(\ww,\ell_p)$. Moreover,
\[
\dist_m(d_1(\ww),\ell_p[\ww])\approx \enpar{\log(1+m)}^{1-1/p}, \quad m\in\NN.
\]
\end{proposition}

\begin{proof}
If $\uu_p$ is the weight with primitive sequence $(m^{1/p})_{m=1}^\infty$ then, by Lemma~\ref{lem:GeneralCoupling}, $\ell_p$ couples with $\uu_p$, and $(\tb,1)$ is congruent con $(\uu_p,\ell_p)$. Taking into account that $\ww$ is equivalent to $\uu_p$, we infer that $\ell_p$ couples with $\ww$, and $(\tb,c)$ is congruent with $(\ww,\ell_p)$. Since $\ell_p[\ww]=\ell_p=d_p(\ww)$, applying Proposition~\ref{lem:ML} puts an end to the proof.
\end{proof}

\section{The construction}\label{sect:main}\noindent
Let $\It$ denote the set of all finite integer intervals contained in $\NN$. Let $\Bt$ be the set of all pairs $(B,I)$, where $B$ is a finite subset of $\NN$ and $I\in\It$ with $\abs{B}<I$, that is, $I=\emptyset$ or $\abs{B}<\min(I)$. Given a weight $\ww=(w_n)_{n=1}^\infty$ and $(B,I)\in\Bt$ we consider the functional
\[
T_{\ww,I,B}=\sum_{n\in I \setminus B} w_j\, \ee_j^*.
\]
If $\ww$ is positive, the family $\{T_{\ww,I,B} \colon (B,I)\in\Bt\}$ separates the vectors of $\FF^\NN$, therefore
\[
b_{\ww}:=\enbrace{f\in\FF^\NN \colon \norm{f}_{b_{\ww}}:=\sup_{(B,I)\in\Bt} \abs{T_{\ww,I,B}(f)}<\infty}
\]
is a Banach space. The wished-for conditional bases with Property~(A) will be the unit vector systems of the Banach spaces
\[
\BB[\ww,\Sym]:=\Sym[\ww] \cap b_{\ww}
\]
associated with non-increasing weights $\ww$ and symmetric Banach spaces $\Sym$ that properly couple with $\ww$. Since Property~(A) is isometric in nature, to properly state this result we need to choose a suitable norm for $\BB[\ww,\Sym]$. Specifically, we choose a constant $c$ and a non-increasing sequence $\tb=(t_m)_{m=1}^\infty$ such that $(\tb,c)$ is congruent with $(\ww,\Sym)$. Then we define
\[
\norm{f}_{\ww,\Sym} =\sup_{(B,I)\in\Bt} t_{\abs{B}} \norm{S_B(f)}_{\Sym} + \frac{c}{2} \abs{T_{\ww,I,B}(f)}.
\]
To avoid cumbrous notations we are omitting the dependence of $\norm{\cdot}_{\ww,\Sym}$ on $\tb$ and $c$.

To state the features of the unit vector system of the spaces $\BB[\ww,\Sym]$, we recall some additional terminology related to the TGA.

A basis $\XB$ is said to be \emph{bidemocratic} if there is a constant $C\in(0,\infty)$ such that
\begin{equation}\label{eq:Bid}
\usdf_m[\XB,\XX] \, \usdf_m[\XB^*,\XX^*] \le Cm,\quad m\in\NN.
\end{equation}
This constant $C$ cannot be smaller than one. If \eqref{eq:Bid} holds with $C=1$ we say that $\XB$ is \emph{isometrically bidemocratic}. If this is the case then both $\XB$ and $\XB^*$ are isometrically superdemocratic, that is,
\[
\norm{\Ind_{\varepsilon,A}[\XB,\XX]}= \usdf_m[\XB,\XX], \quad \norm{\Ind_{\varepsilon,A}[\XB^*,\XX^*]}=\usdf_m[\XB^*,\XX^*]
\]
for all $A\subset\NN$ with $\abs{A}=m$ and all $\varepsilon\in\EE^A$. Notice that, by Proposition~\ref{prop:dualsymmetric}, the unit vector system of any symmetric Banach space is isometrically bidemocratic.

Let $S_A=S_A[\XB,\XX]$ be the coordinate projection onto the linear span of the elements from the basis $\XB$ indexed from the finite set $A\subset\NN$. To quantify how far $\XB$ is from being unconditional we use the sequence $(\unc_m)_{m=1}^\infty$ of unconditionality parameters given by
\[
\unc_m=\unc_m[\XB,\XX]=\sup_{\abs{A}\le m} \norm{S_A[\XB,\XX]}, \quad m\in\NN.
\]
If $\XB$ is a quasi-greedy basis then we know (\cite{DKK2003}*{Lemma 8.2}) that
\[\unc_m[\XB,\XX]\lesssim \log(1+m),\quad m\in\NN.\] We will also use another sequence of unconditionality parameters defined for each $m\in\NN$ by
\[
\uunc_m=\uunc_m[\XB,\XX]=\sup\enbrace{ \norm{S_A(f)} \colon A\subset\NN,\, \norm{f} \le 1, \, \supp(f)\subset[1,m] },
\]
(see \cites{GW2014, AAW2019}).
Note that, for all $m\in\NN$, $\uunc_m\le \unc_m$, and
\begin{equation}\label{eq:UCEquiv}
\sup\{ \Vert f-2 S_A(f) \Vert \colon A\subset\NN,\, \Vert f \Vert \le 1, \, \supp(f)\subset[1,m]\} \le 2 \uunc_m.
\end{equation}

A basis $\XB$ is said to be \emph{bimonotone} if $\norm{S_I[\XB,\XX]}\le 1$ for every integer interval $I\subset \NN$.
Similarly, a map $\norm{\cdot} \colon \FF^\NN\to[0,\infty]$ is said to be bimonotone if $\norm{S_I(f)}\le \norm{f}$ for for every integer interval $I$ and every $f\in\FF^\NN$.

We will use several times the estimate provided by the following lemma.
\begin{lemma}\label{lem:bwdoow}
Let $\ww$ be a non-increasing weight and $(B,I)\in\Bt$. Then,
\begin{equation*}
\abs{T_{\ww,I,B}(f)} \le \abs{T_{\ww,I,\emptyset}(f)} + \norm{f}_{\infty,\ww}, \quad f\in\FF^\NN.
\end{equation*}
\end{lemma}
\begin{proof}
Set $f=(a_n)_{n=1}^\infty$ and $m=\abs{B}$. We have
\[
\abs{T_{\ww,I,B}(f)}\le \abs{T_{\ww,I,\emptyset}(f)} +R,
\]
where
\[
R=\abs{\sum_{n\in I\cap B} a_n \, w_n}\le \norm{f}_{\infty,\ww} \sum_{n\in I\cap B} \frac{w_n}{s_n}
\le \norm{f}_{\infty,\ww} \sum_{n=m+1}^{2m} \frac{w_n}{s_n}\le \norm{f}_{\infty,\ww}.\qedhere
\]
\end{proof}

We will also appeal a couple of times to the following observation.
\begin{lemma}\label{lem:new1}
If $I\in\It$ and $(B,J)\in \Bt$ then $(B,I\cap J)\in \Bt$, and
\begin{equation*}
T_{\ww, J,B} \circ S_I= T_{\ww, I\cap J,B}.
\end{equation*}
\end{lemma}

As we will show, $c_{00}$ is not dense in $\BB[\ww]$. Thus we denote by $\BB_0[\ww]$ the separable part of $\BB[\ww]$.

\begin{theorem}\label{thm:main}
Suppose that a symmetric Banach space $\Sym$ couples with a non-increasing non-summable weight $\ww$, and pick a pair $((t_n)_{n=1}^\infty, c)$ congruent with $(\ww,\Sym)$. Then the unit vector system $\EB$ of the Banach space $\BB:=\BB_0[\ww,\Sym]$ is a Schauder basis with Property~(A). In addition, if $\FuFu[\Sym]=(\Lambda_m)_{m=1}^\infty$ and $d=\sup_m \Lambda_m t_m/s_m$,
\begin{enumerate}[label=(\roman*),leftmargin=*, widest=iii]
\item\label{it:PA:0} The norm $\norm{\cdot}_{\ww,\Sym}$ is bimonotone, so that $\EB$ is a bimonotone basis of $\BB$,
\item\label{it:PA:1} $\EB$ is isometrically bidemocratic,
\item\label{it:PA:2} $\usdf[\EB,\BB] =(t_m\Lambda_m)_{m=1}^\infty\approx \prim[\ww]$,
\item\label{it:PA:3} $\unc_m[\EB,\BB] \lesssim \dist_m(d_1(\ww),\Sym[\ww])$ for $m\in\NN$, and
\item $\uunc_m[\EB,\BB] \ge 4^{-1} (c+d)^{-1} c \, \dist_m(d_1(\ww),\Sym[\ww])$ for $m\in\NN$.

\end{enumerate}
If $\Sym$ properly couples with $\ww$, $\EB$ is a conditional basis of $\BB$.
\end{theorem}

Before tackling the proof of Theorem~\ref{thm:main} we single out a property of the norm on $\BB[\ww,\Sym]$ that will be crucial in showing that the unit vector system of $\BB_0[\ww,\Sym]$ has Property~(A).

\begin{lemma}
Given $f=(a_n)_{n=1}^\infty\in \FF^\NN$ with $\abs{a_n} \le 1$ for all $n\in\NN$, and $A\subset\NN$ finite with $\abs{a_n}=1$ for all $n\in A$,
\[
\norm{f}_{\ww,\Sym}
=\sup_{\substack{(B,I)\in\Bt\\ A\subset B\subset \supp(f)}} t_{\abs{B}} \norm{S_{B}(f)}_{\Sym} + \frac{c}{2} \abs{T_{\ww,I,B}(f)}.
\]
\end{lemma}

\begin{proof}
If $A=\emptyset$ we pick $(B,I)\in \Bt$ and set $D=B\cap \supp(f) $. We have $(D,I)\in\Bt$, and
\begin{align*}
t_{\abs{B}} \norm{S_{B}(f)}_{\Sym} +\frac{c}{2} \abs{T_{\ww,I,B}(f)}
&= t_{\abs{B}} \norm{S_{D}(f)}_{\Sym} +\frac{c}{2} \abs{T_{\ww,I,D}(f)}\\
&\le t_{\abs{D}} \norm{S_{D}(f)}_{\Sym} +\frac{c}{2} \abs{T_{\ww,I,D}(f)}.
\end{align*}
Now, to show the result when $A$ is nonempty, we pick $(B,I)\in \Bt$ with $B\subset\supp(f)$, and $k\in A\setminus B$. Set $D=B\cup\{k\}$, $m=\abs{B}$ and $J=I\setminus \{m+1\}$. We have $(D,J)\in\Bt$ and
\[
E:=(J\setminus D)\triangle (I\setminus B) \subset \enbrace{m+1,k}\cap[m+1,\infty),
\]
therefore
\[
\abs{T_{\ww,I,B}(f) -T_{\ww,J,D}(f)}\le \sum_{n\in E} \abs{a_n} \, w_n\le 2 w_{m+1}.
\]
Hence,
\begin{align*}
t_{\abs{B}} \norm{S_{B}(f)}_{\Sym} + \frac{c}{2} \abs{T_{\ww,I,B}(f)}
&\le t_{\abs{B}} \norm{S_{B}(f)}_{\Sym} + c w_{m+1}+ \frac{c}{2} \abs{T_{\ww,J,D}(f)}\\
&\le t_{\abs{D}} \norm{S_{D}(f)}_{\Sym} + \frac{c}{2} \abs{T_{\ww,J,D}(f)},
\end{align*}
and the proof is over.
\end{proof}

\begin{proof}[Proof of Theorem~\ref{thm:main}]
\ref{it:PA:0} is a ready consequence of Lemma~\ref{lem:new1}.
To see that $\EB$ has Property~(A), we pick $m\in\NN$, $A\subset\NN$ with $\abs{A}=m$, $\varepsilon\in\EE^A$, and $f\in\FF^\NN$ with $\norm{f}_\infty\le 1$ and $\supp(f)\subset\NN\setminus A$. We have
\begin{align*}
\norm{\Ind_{\varepsilon,A}+f }_{\ww,\Sym}
=&\Nr(A,\varepsilon,f)\\
:=&\sup_{\substack{D\subset \supp(f)\\ m + \abs{D} <\min(I) }} t_{m + \abs{D} } \norm{\Ind_{\varepsilon,A}+S_D(f)}_{\Sym}+ \frac{c}{2}T_{\ww,I,D}(f).
\end{align*}
By symmetry (or, more specifically, by Property~(A)), $\Nr(A,\varepsilon,f)$ only depends on $A$ and $\varepsilon$ through $m$. Hence, the unit vector system has Property~(A). In the particular case that $f=0$ we obtain
\[
\norm{\Ind_{\varepsilon,A}}_{\ww,\Sym}= t_{\abs{A}} \norm{\Ind_{\varepsilon,A}}_{\Sym}= t_m \Lambda_m.
\]

Consider $\Sym[\ww]$ equipped with the equivalent norm
\[
\Nnorm{f}_{\Sym[\ww]}
= \sup_{\abs{B}<\infty} t_{\abs{B}}\norm{S_B(f)}_\Sym, \quad f\in\FF^\NN,
\]
so that $\Nnorm{\cdot}_{\Sym[\ww]}\le \norm{\cdot}_{\ww,\Sym}$. By Proposition~\ref{prop:dualsymmetric},
\[
\norm{\Ind^*_{\varepsilon,A}}_{\BB^*} \le \Nnorm{\Ind^*_{\varepsilon,A}}_{(\Sym[\ww])^*}=\frac{\abs{A}}{\Nnorm{\Ind_{\varepsilon,A}}}_{\Sym[\ww]}=\frac{m}{t_m \Lambda_m},
\]
from where we infer that \ref{it:PA:1} and \ref{it:PA:2} hold.

Since $d_1(\ww)\subset b_{\ww}$, by Lemma~\ref{lem:Swnotd1w},
\[
d_1(\ww) \subset \BB[\ww,\Sym] \subset \Sym[\ww].
\]
Therefore, by \cite{AAB2021}*{Lemma 6.1}, \ref{it:PA:3} holds. To estimate from below the unconditionality constants we compare the norm in $\BB[\ww,\Sym]$ of a non-increasing nonnegative sequence $f=(a_n)_{n=1}^\infty\in c_{00}$ with that of $g=((-1)^n a_n)_{n=1}^\infty$. Since $(\emptyset,I)\in\Bt$ for every $I\in \It$,
\[
\norm{f}_{\ww,\Sym}\ge \frac{c}{2} \sup_{I\in \It} \abs{T_{\ww,\emptyset,I}(f)}
= \frac{c}{2} \sum_{n=1}^\infty a_n w_n =\frac{c}{2} \norm{f}_{1,\ww}.
\]
Pick $(B,I)\in\Bt$ with $I\not=\emptyset$, and set $\abs{B}=m$ and $\min(I)=k$. Using Lemma~\ref{lem:bwdoow} we obtain
\[
\abs{T_{\ww,I,B}(g) } \le a_k w_k + \norm{g}_{\infty,\ww}
\le a_k s_k+ \norm{g}_{\infty,\ww}
\le 2 \norm{f}_{\infty,\ww}.
\]
Consequently, by Lemma~\ref{lem:embedding}, $\norm{g}_{\ww,\Sym} \le (d+c ) \norm{f}_{\Sym[\ww]}$. By \eqref{eq:UCEquiv},
\[
\uunc_m[\EB,\BB] \ge \frac{c}{4(c+d)} \dist_m(d_1(\ww), \Sym[\ww]), \quad m\in\NN.
\]
Appealing to the `moreover' part of Lemma~\ref{lem:Swnotd1w}, the proof is over.
\end{proof}

The combination of Theorem~\ref{thm:main} with Lemma~\ref{lem:GeneralCoupling}, Proposition~\ref{lem:L1coupleswith} or Proposition~\ref{lemma:lp} provides an adaptable method for constructing Banach spaces with conditional bases having Property~(A). Let us illustrate our construction with a pair of such Banach spaces.

\begin{example}\label{example:lpv1}
Let $1<p<\infty$. For $f=(a_n)_{n=1}^\infty\in\FF^\NN$ define
\[
\norm{f}=\sup_{ \substack{ B\subset \NN\\ \abs{B}<\infty }} \enpar{\abs{B}^{1/p-1} \sum_{n\in B} \abs{a_n} + \sup_{\substack{I \in\It \\ I>\abs{B} }} \abs{\sum_{n\in I\setminus B} a_n \frac{n^{1/p}-(n-1)^{1/p}}{2}}}.
\]
Let $\XX$ be the subspace of $\{f\in\FF^\NN \colon \norm{f}<\infty\}$ spanned by $c_{00}$. By Theorem~\ref{thm:main} and Proposition~\ref{lem:L1coupleswith}, the unit vector system $\EB$ is a bimonotone basis of $\XX$. Besides, $\EB$ has Property~(A), is isometrically bidemocratic with $\usdf_m[\EB,\XX]=m^{1/p}$ for all $m\in\NN$, and conditional with $\unc_m[\EB,\XX]\approx \uunc_m[\EB,\XX]\approx \log(1+m)$ for $m\in\NN$.
\end{example}

\begin{example}
Let $1<p<\infty$. Define for $f=(a_n)_{n=1}^\infty\in\FF^\NN$
\[
\norm{f}=\sup_{ \substack{ B\subset \NN\\ \abs{B}<\infty }} \enpar{\enpar{ \sum_{n\in B} \abs{a_n}^p}^{1/p} + \frac{1}{2p}\sup_{\substack{I \in\It \\ I>\abs{B}}}
\abs{\sum_{n\in I\setminus B} a_n n^{1/p-1}}}.
\]
Let $\XX$ be the subspace of $\{f\in\FF^\NN \colon \norm{f}<\infty\}$ spanned by $c_{00}$. Since
\[
\inf_n\frac{n^{1/p}-(n-1)^{1/p}}{n^{1/p-1}}=\lim_{x\to 0^+} \frac{1-(1-x)^{1/p}}{x}=\frac{1}{p},
\]
combining Theorem~\ref{thm:main} with Proposition~\ref{lemma:lp} gives that the unit vector system $\EB$ is a bimonotone basis of $\XX$. Besides, $\EB$ has Property~(A), is isometrically bidemocratic with $\usdf_m[\EB,\XX]\approx m^{1/p}$ for all $m\in\NN$, and conditional with $\unc_m[\EB,\XX] \approx \uunc_m[\EB,\XX]\approx \log^{1/p}(1+m)$ for $m\in\NN$.
\end{example}

More generally, we have the following result.

\begin{corollary}
There exist Banach spaces $\XX$ with a bidemocratic Schauder basis $\XB$ that fails to be unconditional and yet has Property~(A). In fact, if we let $\prim=(s_m)_{m=1}^\infty$ be an essentially increasing sequence of positive scalars whose dual sequence $\prim^*$ is also essentially increasing such that
\begin{enumerate}[label=(C.\arabic*)]
\item\label{it:cor:1} $\prim$ and $\prim^*$ are unbounded, and $\prim$ is equivalent to the primitive sequence of some non-increasing weight $\ww$,
\end{enumerate}
then we can choose $\XX$ and $\XB$ so that
\[
\usdf_m[\XB,\XX] =\prim[\ww]\approx s_m, \quad m\in\NN.
\]
In particular, such Banach spaces can be chosen if
\begin{enumerate}[label=(C.\arabic*),resume]
\item\label{it:cor:2} $\prim$ has the LRP and $\prim^*$ is unbounded.
\end{enumerate}
Finally, if
\begin{enumerate}[label=(C.\arabic*),resume]
\item\label{it:cor:3} $\prim$ has the LRP and the URP,
\end{enumerate}
then we can find $\XX$ and $\XB$ so that, in addition to the above properties,
\[
\unc_m[\XB,\XX] \approx \uunc_m[\XB,\XX] \approx \log(m+1),\quad m\in\NN.
\]
\end{corollary}

\begin{proof}
If \ref{it:cor:1} holds, $\ww$ is non-summable and null. If $\prim$ has the LRP, it is unbounded and equivalent to the primitive weight of $(s_n/n)_{n=1}^\infty$ (see \cite{AlbiacAnsorena2016}*{Lemma 2.12}). Therefore, \ref{it:cor:3} implies \ref{it:cor:2}, and \ref{it:cor:2} implies \ref{it:cor:1}. So, the result follows by combining Proposition~\ref{lem:L1coupleswith} with Theorem~\ref{thm:main}.
\end{proof}

To finish the section we state a straightforward corollary which solves \cite{AlbiacAnsorena2017b}*{Problem 4.4}.

\begin{corollary}
There are Banach spaces with $1$-almost greedy bases which are not $C$-greedy for any $C\ge 1$.
\end{corollary}

\section{Further applications}\label{SectNew}\noindent
The study conducted in \cite{AABHO2024} about the connections between Property~(A) and the unconditionality of bases in the case when the unconditional constants are sharp culminated in the aforementioned Theorem H and led naturally to wonder whether unconditional bases with Property~(A) have suppression unconditionality constant $K_s\le C$ for some absolute constant $C>1$ (see \cite{AABHO2024}*{Question 3.10} and the preceding discussion). Our construction allows us to answer this question in the negative.

\begin{lemma}\label{lem:trivial}
Given an arbitrary non-increasing weight $\ww$, $d_1(\ww)$ couples with $\ww$, and $\XX:=\BB[d_1(\ww), \ww]$ is a renorming of $d_1(\ww)$. Moreover, the pair $((1)_{m=1}^\infty, 1)$ is congruent with $(\ww,d_1(\ww))$.
\end{lemma}

\begin{proof}
Clearly, $\abs{T_{\ww,I,B}(f)} \le \norm{f}_{1,\ww}$ for all $f\in\FF^\NN$, and $\FuFu[d_1(\ww)]=\prim[\ww]$. By Lemma~\ref{lem:embedding}\ref{it:embedding:A}, $\XX=d_1(\ww)$. Applying Lemma~\ref{lem:GeneralCoupling} with $q=1$ and $\uu=\ww$ puts an end to the proof.
\end{proof}

\begin{theorem}
For each $C\ge 1$ there is a renorming of $\ell_1$ with respect to which the unit vector system $\EB$ has Property~(A) while fails to be $C$-suppression unconditional. Besides, we can choose this renorming so that $\EB$ is isometrically bidemocratic and bimonotone.
\end{theorem}

\begin{proof}
Pick an arbitrary non-increasing, non-summable, null weight $\ww=(w_n)_{n=1}^\infty$. For each $m\in\NN$ put
\[
\ww_m =\enpar{ w_{\min\{n,m\}}}_{n=1}^\infty.
\]
Since $\ww_m$ fails to be null, by Lemma~\ref{lem:trivial},
\[
\BB_m:=\BB[\ell_1,\ww_m]=\ell_1[d_1(\ww_m),\ww_m]=d_1(\ww_m)=\ell_1.
\]
By Lemma~\ref{lem:trivial} and Theorem~\ref{thm:main}, the space $\ell_1$ equipped with the norm of $\BB_m$ has Property~(A). Now, if the unit vector system were $C$-suppression unconditional with respect to that norm, we would have
\[
C \ge \frac{1}{8} \dist_m(d_1(\ww_m),\ell_1[\ww_m]).
\]
Since $\norm{f}_{\ell_1[\ww_m]}=\norm{f}_{\ell_1[\ww]}$ and $\norm{f}_{1,\ww_m}=\norm{f}_{1,\ww}$ when $f$ has at most $m$ nonzero coordinates,
\[
C\ge \frac{1}{8} \dist_m(d_1(\ww),\ell_1[\ww])= C_m:= \frac{1}{8} \dist_m(d_1(\ww),m(\ww)).
\]
Since $d_1(\ww) \not=m(\ww)$, $\lim_m C_m=\infty$.
\end{proof}

Let us next unearth the discussion about the relation between unconditionality and quasi-greediness in the isometric case. This question was partially addressed in \cite{AlbiacAnsorena2016c}, where it was shown that, oddly enough, being $1$-suppression unconditional is equivalent to being $1$-quasi-greedy, i.e.,
\[
\norm{\Gt_m(x)}\le \norm{x}, \quad x\in\XX, \, m\in\NN,
\]
(see Theorem C above). However, it was left open whether being $1$-suppression quasi-greedy, i.e., verifying
\[
\norm{x-\Gt_m(x)} \le \norm{x}, \quad x\in\XX, \, m\in\NN,
\]
also implies $1$-suppression unconditionality, or at least being unconditional. Since a basis with Property~(A) is $1$-almost greedy and therefore $1$-suppression quasi-greedy, Theorem~\ref{thm:main} answers in the negative this question, which was first raised in \cite{AlbiacAnsorena2016c}*{Problem 3.4} and recently reformulated in \cite{AAT2024}*{comments after Problem~6}). A standard truncation argument will allow us to go a bit further by improving the geometry of the underlying space.

\begin{theorem}
There is a reflexive Banach space $\XX$ with a conditional, bimonotone basis $\XB$ that is isometrically bidemocratic, $1$-suppression quasi-greedy, and satisfies
\[
\unc_m[\XB,\XX] \approx \uunc_m[\XB,\XX] \approx \log(1+m), \quad m\in\NN.
\]
\end{theorem}

\begin{proof}
Fix $1<p<\infty$. By Example~\ref{example:lpv1} there is a Banach space $\YY$ with a basis $\YB=(\yy_n)_{n=1}^\infty$ such that
\begin{itemize}[leftmargin=*]
\item $\YB$ is bimonotone,
\item $\YB$ is isometrically bidemocratic with $\usdf_m[\YB,\YY]=m^{1/p}$ for all $m\in\NN$,
\item $\YB$ is $1$-suppression quasi-greedy.
\item $\uunc_m[\YB,\YY] \approx \log(1+m)$ for $m\in\NN$.
\end{itemize}
For each $k\in\NN$, let $\YB_k=(\yy_n)_{n=1}^{2^{k-1}}$ and $\YY_k=\spn(\YB_k)$. The natural arrangement of
\[
\XB=\oplus_{k=1}^\infty \YB_k
\]
is a basis of $\enpar{\oplus_{k=1}^\infty \YY_k}_{\ell_p}$ that inherits from $\YB$ all the above mentioned properties.
\end{proof}

\section{The geometry of \texorpdfstring{$\BB[\ww,\Sym]$}{}}\label{Sect4}\noindent
In this section we study some properties of the Banach space $\BB[\ww,\Sym]$ constructed in Section~\ref{sect:main}.
The only restriction on the weight $\ww$ and the symmetric Banach space $\Sym$ we now assume a priori is $(s_m/\Lambda_m)_{m=1}^\infty$ to be bounded so that we can construct the symmetric Banach space $\Sym[\ww]$.
Replacing $\Sym[\ww]$ with $\Sym$ we can assume that $\Sym[\ww]=\Sym$ and that $\Lambda_n\approx s_n$ for all $n\in\NN$. We will consider on $\BB[\ww,\Sym]$ the norm
\[
\Nnorm{f}_{\ww,\Sym}=\max \enbrace{ \norm{f}_{\Sym}, \norm{f}_{b_{\ww}}}.
\]
By Lemma~\ref{lem:new1}, this norm is also bimonotone.

\begin{lemma}\label{lem:uvbNBC}
Let $\ww$ be a non-increasing weight and $\Sym$ be a symmetric Banach space with $\FuFu[\Sym]\approx\prim[\ww]$ and $d_1(\ww) \subsetneq \Sym$. There exist pairwise disjoint intervals $(J_k)_{k=1}^\infty$ in $\It$, and vectors $(f_k)_{k=1}^\infty$ and $(g_A)_{A\subset\NN}$ in $\BB[\ww,\Sym]$ such that
\begin{itemize}
\item $\supp(f_k)=J_k$ for all $k\in\NN$,
\item $\inf\{\Nnorm{f_k}_{\ww,\Sym} \colon k\in\NN\}>0$,
\item $\supp(g_A)=\cup_{k\in A} J_k$ for all $A\subset\NN$,
\item $\sup\{ \Nnorm{g_A}_{\ww,\Sym} \colon A\subset\NN\}<\infty$, and
\item for each $A\subset\NN$ there is $(\varepsilon_n)_{n\in A}$ in $\EE$ such that
$S_{J_k}(g_A)=\varepsilon_k f_k$ for all $k\in A$.
\end{itemize}
\end{lemma}

\begin{proof}
Given $A\subset\NN$ and $k\in A$, set $\varepsilon_{k,A}=1$ if $k$ is in an even place within $A$ and $\varepsilon_{k,A}=-1$ if $k$ is in an odd place within $A$. Pick and non-increasing positive sequence $f=(a_n)_{n=1}^\infty$ with
\[
\norm{f}_{1,\ww}=\sum_{n=1}^\infty a_n \, w_n=\infty
\]
and $\norm{f}_{\Sym[\ww]}<\infty$. We have
\[
a_n \, w_n \le \norm{f}_{\infty,\ww} \frac{w_n}{s_n} \le \frac{1}{n} \norm{f}_{\infty,\ww}, \quad n\in\NN.
\]
Since $f\in d_\infty(\ww)$, $\lim_n a_n \, w_n=0$. Consequently, we can recursively construct a sequence $(J_k)_{k=1}^\infty$ in $\It$ such that $J_k\not=\emptyset$ and $\max(J_k)<\min(J_{k+1})$ for all $k\in\NN$, and, if we set
\[
\lambda_k=\sum_{n\in J_k} a_n \, w_n, \quad k\in\NN,
\]
$\lambda:=\sum_{k=1}^\infty \abs{\lambda_k-\lambda_{k+1}}<\lambda_1$. Note that this implies
\begin{equation*}
\abs{\sum_{k\in A} \varepsilon_{k,A} \lambda_k} \le \lambda
\end{equation*}
for all $A\subset\NN$ finite and even. Consequently,
\begin{itemize}[leftmargin=*]
\item $ \lambda_k\ge D:=\lambda_1-\lambda$ for all $k\in\NN$, and
\item $\abs{\sum_{k\in A} \varepsilon_{k,A} \lambda_k} \le E:= \lambda_1+\lambda$ for all $A\subset\NN$ finite. In particular,
\item $\abs{\lambda_k} \le E$ for all $k\in\NN$.
\end{itemize}

Set for each $k\in\NN$ $f_k=\sum_{n\in J_k} a_n\, \ee_n$. We have
\[
\Nnorm{f_k}_{\ww,\Sym} \ge T_{\ww,J_k,\emptyset} \enpar{f_k} = \lambda_k \ge D, \quad k\in\NN.
\]

Define $g_A=(b_n)_{n=1}^\infty \in\FF^\NN$ by $b_n=\varepsilon_k \, a_n$ if $n\in J_k$ for some $k\in A$ and $b_n=0$ otherwise. Pick $(B,I)\in\Bt$ and set
\[
L=\enbrace{k\in A \colon J_k\subset I}, \quad K=\enbrace{k\in A \colon J_k\cap I \not=\emptyset}.
\]
Since $\abs{K\setminus L}\le 2$ and $L$ is an interval relative to $A$, applying Lemma~\ref{lem:bwdoow} we obtain
\begin{align*}
\abs{T_{\ww,B,I}(g_A)} & \le \norm{g_A}_{\infty,\ww}+ \abs{\sum_{k\in A} \varepsilon_{k,A} \sum_{n\in I\cap J_k} a_n w_n}\\
&= \norm{g_A}_{\infty,\ww}+ \abs{\sum_{k\in K} \varepsilon_{k,A} \sum_{n\in I\cap J_k} a_n w_n}\\
&\le \norm{g_A}_{\infty,\ww} + 2 E + \abs{\sum_{k\in L} \varepsilon_{k,A} \sum_{n\in I\cap J_k} a_n w_n}\\
&=\norm{g_A}_{\infty,\ww} + 2 E + \abs{\sum_{k\in L} \varepsilon_{k,L} \lambda_k}\\
&\le \norm{f}_{\infty,\ww} +3E.
\end{align*}
By Lemma~\ref{lem:embedding}, $\Nnorm{g_A}_{\ww,\Sym} \le \norm{f}_{\Sym[\ww]} +3E$.
\end{proof}

\begin{proposition}\label{prop:FBC}
Let $\ww$ be a non-increasing weight and $\Sym$ be a symmetric Banach space with $\FuFu[\Sym]\approx\prim[\ww]=(s_n)_{n=1}^\infty$ and $d_1(\ww) \subsetneq \Sym$. Then
\begin{enumerate}[label=(\roman*)]
\item\label{it:FBC:a} The Banach space $\BB[\ww,\Sym]$ is not separable,
\item\label{it:FBC:dos} the unit vector system is not a boundedly complete basis of the separable part $\BB_0[\ww,\Sym]$ of $\BB[\ww,\Sym]$,
\item\label{it:FBC:c} $c_{00}$ is not dense in $\BB[\ww,\Sym]$, and
\item
the unit vector system fails to be a shrinking basis of $\BB_0[\ww,\Sym]$.
\end{enumerate}
\end{proposition}

\begin{proof}
We will use the terminology of Lemma~\ref{lem:uvbNBC}. Let $A$ and $B$ be different subsets of $\NN$, and pick $k\in A \triangle B$. By bimonoticity,
\[
\Nnorm{g_A-g_B}_{\ww,\Sym} \ge \Nnorm{S_{J_k}(g_A-g_B)}_{\ww,\Sym}=\Nnorm{f_k}_{\ww,\Sym}.
\]
Hence, \ref{it:FBC:a} holds.

Set $g:=g_\NN=(b_n)_{n=1}^\infty$. By bimonoticity,
\[
\sup_{m\in\NN}\norm{\sum_{n=1}^m b_n\, \ee_n}_{\ww,\Sym}<\infty.
\]
Since $\sum_{n=1}^\infty b_n\, \ee_n$ does not converge, \ref{it:FBC:dos} holds. We can deduce \ref{it:FBC:c} either from \ref{it:FBC:a} or from the fact that $g\in \BB[\ww,\Sym]\setminus \BB_0[\ww,\Sym]$.

Since $\ww$ is non-summable, we can recursively construct an increasing sequence $(n_k)_{k=0}^\infty$ such that $n_0=0$ and
\[
s_{n_k}-s_{n_{k-1}}=\sum_{n=1+n_{k-1}}^{n_k} w_k \ge \frac{1}{2} s_{n_k}.
\]
Now, on one hand, since $\ww\in \BB_0^*[\ww,\Sym]$, the block basic sequence $(f_k)_{k=1}^\infty$ defined by
\[
f_k=\frac{1}{s_{n_k}} \Ind_{A_k}, \quad A_k=(n_{k-1},n_k]\cap \ZZ,
\]
is not weakly null. On the other hand, since $(f_k)_{k=1}^\infty$ is bounded in $d_1(\ww)$, it is also bounded in $\BB[\ww,\Sym]$.
\end{proof}

\begin{proposition}\label{prop:NotReflexive}
Let $\ww$ be a non-increasing weight and $\Sym$ be a symmetric Banach space. Then $\BB_0[\ww,\Sym]$ is not reflexive.
\end{proposition}

\begin{proof}
If $d_1(\ww) \subsetneq \Sym[\ww]$, the result follows from Proposition~\ref{prop:FBC}\ref{it:FBC:dos} and a classical result from \cite{James1950} (see \cite{AlbiacKalton2016}*{Theorem 3.2.19}). If $\Sym[\ww] = d_1(\ww)$, then $\BB_0[\ww,\Sym]=d_1(\ww)$. By Proposition~\ref{lem:ML}\ref{it:ML:6}, $\BB_0^*[\ww,\Sym]=\ell_1[\ww]$. Hence, the result follows from parts~\ref{it:embedding:A} and \ref{it:embedding:B} of Lemma~\ref{lem:embedding}.
\end{proof}

Recall that a basic sequence in a Banach space is said to be \emph{complemented} if its closed linear span is complemented.
\begin{proposition}
Let $\ww=(w_n)_{n=1}^\infty$ be a null weight and $\Sym$ be a symmetric Banach space with $\FuFu[\Sym]=(\Lambda_n)_{n=1}^\infty\approx\prim[\ww]$. Then the unit vector system of $\BB[\ww,\Sym]$ has a complemented subbasis equivalent to the unit vector system of $\Sym$.
\end{proposition}

\begin{proof}
Let $\nu\colon\NN\to\NN$ be an increasing map such that
\[
C:=\sum_{j=1}^\infty w_{\nu(j)}\le \Lambda_1.
\]
Let $L\colon \FF^\NN \to \FF^\NN$ be the canonical lifting given by
\[
(b_k)_{k=1}^\infty\mapsto (a_n)_{n=1}^\infty, \quad a_n=\begin{cases} b_k & \mbox{ if } n=\nu(k) \mbox{ for some } k\in\NN,\\ 0 & \mbox{ otherwise.}\end{cases}
\]

On the one hand, given $f\in\FF^\NN$ we have
\[
\Nnorm{L(f)}_{\ww,\Sym}
\le \max\enbrace{\norm{L(f)}_{\Sym} , C \norm{f}_\infty}
= \max\enbrace{\norm{f}_{\Sym} , C \norm{f}_\infty}
\le \norm{f}_{\Sym}.
\]

On the other hand, the corresponding canonical projection $S\colon\FF^\NN\to \FF^\NN$ given by
\[
(a_n)_{n=1}^\infty\mapsto (a_{\nu(k)})_{k=1}^\infty,
\]
satisfies
\[
\norm{S(f)}_{\Sym} \le \norm{f}_{\Sym} \le \Nnorm{f}_{\ww,\Sym}, \quad f\in\FF^\NN.
\qedhere
\]
\end{proof}

\section{Open questions}\label{OQ}\noindent
From a functional analytic point of view, one of the main topics in greedy approximation theory with respect to bases is to determine the greedy-like basis structure of a given Banach space. In our case, once we know that there exist conditional bases with Property~(A) it is natural to wonder whether these bases can be found in a given Banach space or in a certain class of Banach spaces. Bearing in mind Theorem~\ref{prop:NotReflexive}, the following question arises.

\begin{question}\label{question:Reflexive}
Is there a reflexive Banach space with a conditional Schauder basis that has Property~(A)?
\end{question}

To address Question~\ref{question:Reflexive} it may be interesting to focus on specific spaces or more restrictive classes of spaces. Understanding the structure of conditional bases with Property~(A) of Hilbert spaces should be a primary task.
\begin{question}\label{question:Hilbert}
Is there a conditional Schauder basis with Property~(A) for a renorming of $\ell_2$?
\end{question}

We point out that allowing renormings is essential for Question~\ref{question:Hilbert} to make proper sense. In fact, if a basis $\XB=(\xx_n)_{n=1}^\infty$ of $\ell_2$ has Property~(A) then
\begin{equation}\label{eq:PLaw}
\norm{\xx_n+\xx_k}_2=\norm{\xx_n+\varepsilon\, \xx_k}_2, \quad n, k\in\NN, \, n\not=k, \, \varepsilon \in\EE.
\end{equation}
By the polarization identity, \eqref{eq:PLaw} is equivalent to $\XB$ being orthogonal.


\end{document}